\documentclass[12pt]{article}

\usepackage{latexsym,amsfonts,amsmath,amssymb,epsfig,tabularx,amsthm,dsfont,mathrsfs}
\usepackage{graphicx}
\usepackage{hyperref}
\usepackage{enumerate}
\usepackage{color}

\allowdisplaybreaks

\theoremstyle{plain}

\newtheorem{thm}{Theorem}[section]
\newtheorem{rem}[thm]{Remark}

\newtheorem{lem}[thm]{Lemma}
\newtheorem{prop}[thm]{Proposition}
\newtheorem{cor}[thm]{Corollary}
\newtheorem{example}[thm]{Example}
\newtheorem{ass}[thm]{Assumption}
\newtheorem{defi}[thm]{Definition}

\renewcommand{\d}{{\rm d}}

\newcommand{\norm}[1]{\left\Vert #1 \right\Vert}

\newcommand{\N}{\mathbb{N}}
\newcommand{\E}{\mathbb{E}}
\newcommand{\R}{\mathbb{R}}

\newcommand{\cX}{\mathcal{X}}

\DeclareMathOperator{\eps}{\varepsilon}

\DeclareMathOperator{\Id}{\text{Id}}

\title{
Pathwise stability of Wasserstein contractive Markov chain Monte Carlo
}

\author{Julian Hofstadler\thanks{Email: jh4272@bath.ac.uk}\\[1ex]University of Bath}

\date{\today}

\begin{document}
	\maketitle

\begin{abstract}
\noindent
We study pathwise behaviour of Markov chain Monte Carlo (MCMC) algorithms. For chains satisfying a contraction assumption w.r.t.~a Wasserstein distance we establish explicit deterministic rates that hold for single realisations on sets of high probability. Additionally, we show that the median of means estimator further increases the pathwise stability. To prove our main results we employ a well-known martingale decomposition technique based on solutions of Poisson's equation. Finally, we apply our results to concrete MCMC methods including the No-U-Turn-Sampler, the Metropolis adjusted Langevin algorithm and slice samplers. 
\end{abstract}

\noindent
\textbf{MSC classification:} 60J05, 60J22, 65C05

\noindent
\textbf{Keywords:} Markov chain Monte Carlo, finite-sample error, maximal inequality, Poisson's equation, median of means estimator

\section{Introduction}\label{sec:Intro}

Markov chain Monte Carlo (MCMC) algorithms are popular tools in statistics and the natural sciences. Given a probability space $(\cX, \mathcal{F}_\cX, \pi)$ and some sufficiently regular $f \colon \cX \to \R$ these allow to estimate integrals  
\begin{equation}\label{equ:integral_of_interest}
\pi(f) = \int_\cX f(x) \pi(\d x).
\end{equation}
The idea behind this approach is to simulate a Markov chain $(X_n)_{n \in \N}$ targeting $\pi$ in a suitable sense. Then, a proxy for $\pi(f)$ is given by 
\[
\frac{1}{n}\sum_{j=1}^n f(X_j) =: \frac{1}{n}S_n f.
\]

Under mild conditions $\frac{1}{n}S_n f$ converges almost surely to $\pi(f)$, cf. \cite{asmussen2011ergodic}. However, to guarantee efficiency in practice it is crucial to understand the non-asymptotic behaviour of $\frac{1}{n}S_n f$. Existing works establish, for example,  
\begin{equation}\label{equ:MSE_bound}
\E\left[\left\vert \frac{1}{n} S_n f - \pi(f) \right\vert^2  \right] \leq \frac{C_f}{n},
\end{equation}
for $n \in \N$, we refer e.g.~to \cite{belloni2009computational,rudolf2009explicit,rudolf2012explicit,hofstadler2025optimal} in the context of $L^p$ functions, or \cite{Joulin2010Curvature} for Lipschitz functions, see also \cite{latuszynski2015nonasymptotic}. The special case of finite $\cX$ was studied e.g.~in \cite{rudolf2010error}. Furthermore, concentration inequalities, such as Hoeffding's inequality \cite{GLYNN2002hoeffding,MIASOJEDOW2014hoeffding} see also \cite{adamczak2008tail,adamczak2015exponential}, show that $\frac{1}{n}S_n f$ approximates $\pi(f)$ well with high probability for fixed $n \in \N$. 

In this paper we are interested in bounds of the form 
\begin{equation}\label{equ:path_wise_bound}
\left\vert \frac{1}{n} \sum_{j=1}^n f(X_j(\omega)) - \pi(f)\right\vert \leq {C(\omega)}{r_n},
\end{equation}
where $(\Omega, \mathcal{F}, \mathbb{P})$ is the underlying probability space of $(X_n)_{n \in \N}$ and $\omega \in \Omega$ is considered fixed. This allows to quantify the speed of convergence of single realisations of MCMC algorithms.

The law of the iterated logarithm (LIL), if it holds, implies \eqref{equ:path_wise_bound} with $r_n = \frac{\log(\log(n))}{\sqrt{n}}$ and a `random constant' $C(\omega)$. In general, this rate cannot be improved. The LIL is known e.g.~for $V$-uniformly ergodic chains \cite{meyn2009markov} or under Wasserstein contraction assumptions \cite{bolt2012invariance}. See also \cite{hofstadler2024convergenceRates} for similar results in the context of adaptive MCMC. Clearly, `random constants' in \eqref{equ:path_wise_bound} are a drawback from a practical perspective and we are not aware of results with explicit expressions/bounds for $C(\omega)$. 

The present paper addresses this issue. Our main results show that on sets of `large probability' \eqref{equ:path_wise_bound} holds with deterministic rates. For fixed $N \in \N$ and $n \geq N$, we show \eqref{equ:path_wise_bound} with $C(\omega)=1$ and 
\begin{itemize}
	\item[1)] the rate $r_n = \frac{\log(n)}{\sqrt{n}}$ with probability at least $1- \frac{c}{\log(N)}$, see Corollary~\ref{cor:almost_unif_conv_good_rate},
	\item[2)] the rate $r_n = n^{-1/2+\alpha}$ where $\alpha \in (0, 1/2)$ with probability at least $1- \frac{c}{N^{2\alpha}}$, see Corollary~\ref{cor:almost_unif_conv_slow_rate}.
\end{itemize}
The constant $c$ is known, see Remark \ref{rem:computation_of_constant} and by setting $\widetilde{r_n} = \sqrt{c}r_n$ the lower bounds for the probability become independent of $c$, see Remark \ref{rem:tractable_bound_for_probability}. 

Additionally, we study the `median trick', see \cite{niemiro2009fixed, latuszynski2011rigorous}, based on $\ell$ (odd) iid copies of $\frac{1}{n}S_nf$. Our main result in this setting shows that the corresponding estimator $M_{n,\ell}f$ achieves the same rate as $\frac{1}{n}S_nf$ while the probability of having $C(\omega)=1$ improves exponentially in $\ell$, see Corollary~\ref{cor:median_as_convergence}. 

In practice, our results can be used to tune MCMC algorithms. Namely, one may determine $N$ (and $\ell$ if the median is used) such that single runs of MCMC simulations converge sufficiently well and reach a desired accuracy level with high probability. See below Corollary \ref{cor:median_as_convergence} for more details regarding this topic.

Let us discuss our main assumptions and sketch the idea of the proof. Similar as in \cite{bolt2012invariance} we assume a contraction condition w.r.t.~a Wasserstein distance $\mathcal{W}$ and consider Lipschitz functions within \eqref{equ:integral_of_interest}. However, we do not assume that the metric in $\mathcal{W}$ renders $\cX$ a Polish space. Thus we recover uniformly and $V$-uniformly ergodic chains and further examples which are of relevance in the MCMC literature and applications, see Section \ref{sec:Wasserstein_contractive_chains}.  

We aim to employ a well-known martingale approximation technique. To this end, given a suitable $f\colon \cX \to \R$, we need to find $u \colon \cX \to \R$ such that  
\begin{equation}\label{equ:Poisson_equation}
	u(x) - Pu(x) = f(x) - \pi(f),
\end{equation}
where $P$ is the Markov kernel of $(X_n)_{n \in \N}$ and $\pi$ the invariant distribution. Equation \eqref{equ:Poisson_equation} is called \textit{Poisson's equation}, $f$ the \textit{forcing function}, and $u$ a \textit{solution} of Poisson's equation. For any such $u$ we have 
\begin{equation}\label{equ:martingale_decomposition}
	S_n f = \sum_{j=1}^n\left(  u(X_j) - Pu(X_j)\right) = M_n + u(X_1) -u(X_{n+1}), 
\end{equation}
where $M_n = \sum_{j=1}^n\left(  u(X_{j+1}) - Pu(X_j)\right)$.
If $u(X_{j+1}) - Pu(X_j)$ is integrable for any $j\in \N$, then, $(M_n)_{n \in \N}$ is a martingale w.r.t.~$\mathcal{F}_j = \sigma(X_k \colon 1 \leq k \leq j+1)$.

The martingale decomposition \eqref{equ:martingale_decomposition} is a popular tool, see \cite{douc2018markov,GLYNN2002hoeffding,meyn2009markov,bolt2012invariance,maigret1978theoreme} and \cite{neveu1972potentiel,revuz1984markov,nummelin1991poisson} for early works on Poisson's equation.

Often, e.g. in \cite{nummelin1991poisson,GLYNN2002hoeffding,meyn2009markov,bolt2012invariance,grama2018limit,herve2025computable, hofstadler2024convergenceRates}, $u$ is computed as pointwise limit of the series
\begin{equation}\label{equ:PE_series}
{u}(x) = 
\sum_{k=0}^\infty\left(  P^k f(x ) - \pi(f)\right). 
\end{equation}
In particular, the works \cite{bolt2012invariance,hofstadler2024convergenceRates} obtain $u$ as in \eqref{equ:PE_series} under a Wasserstein contraction condition for $f$ being Lipschitz. Using very similar assumptions, our approach recovers this via the Neumann series for $(\Id -P)^{-1}$, see Section~\ref{sec:Solutions}. Moreover, we show some regularity properties for $u$ which were not explicitly stated in \cite{bolt2012invariance, hofstadler2024convergenceRates} and allow to keep track of the constants in our main results. Furthermore, this may be of interest also for analysing adaptive MCMC methods, where $u$ is required for whole families of Markov kernels, see e.g.~\cite{andrieu2006ergodicity,saksman2010ergodicity,atchade2013Bayesian,hofstadler2024convergenceRates,laitinen2024invitation}. 

Finally, building upon \eqref{equ:martingale_decomposition} we prove maximal inequalities in Section \ref{sec:max_inequalities} which are our main tool to establish pathwise stability for $\frac{1}{n}S_nf$. As a by-product we obtain an $L^2$ maximal inequality, which implies a bound as in \eqref{equ:MSE_bound} and some further stability for whole simulations up to a time horizon $n$.

We summarise the main contributions of the present paper as follows: 
\begin{itemize}
	\item[1)] We establish pathwise stability as in \eqref{equ:path_wise_bound} under a Wasserstein contraction condition, see Section \ref{subsec:path_wise_convergence_theory}. Both, the rate and constant, are explicitly known on sets of `high' probability.
	\item[2)] We prove that using the `median trick' significantly increases the set of `high probability' on which single realisations are bounded by deterministic rates, cf. Corollary \ref{cor:median_as_convergence} and the discussion below. 
	\item[3)] We apply our results to concrete MCMC algorithms, including the No-U-Turn-Sampler, Metropolis adjusted Langevin algorithms, and different slice samplers, see Section \ref{subsec:applications}. 
\end{itemize}

The remainder of this paper is organised as follows: In Section~\ref{sec:Wasserstein_contractive_chains} we collect some preliminaries and examples of contractive MCMC methods. Section~\ref{sec:Solutions} contains our results about existence and regularity of Poisson's equation. We show maximal inequalities in Section \ref{sec:max_inequalities} before turning to our main pathwise stability results (including applications) in Section \ref{sec:path_wise}.

\section{Wasserstein contractive Markov chains}\label{sec:Wasserstein_contractive_chains}
In this section we review concepts from the Markov chain literature and related fields which are required later.
In particular, we provide a definition of \textit{Wasserstein contractive Markov chains}, cf. Definition \ref{def:contractive_chain}. 
At the end of this section we include a number of examples, which fit into our framework.

Throughout all sections $(\cX , \rho)$ is a (non-trivial) Polish space and $\mathcal{B}(\cX)$ is the corresponding Borel $\sigma$-algebra. Moreover, $d \colon \cX \times \cX \to \R^+$ is a generic lower semi-continuous (w.r.t.~the product topology of $\rho$) metric.

\subsection{Preliminaries and notation}
Let $(X_n)_{n \in \N}$ be a Markov chain evolving in $\cX$ with transition kernel $P\colon \cX \times \mathcal{B}(\cX) \to [0,1]$. 
Unless specified otherwise, $\nu$ denotes the initial distribution, that is $X_1 \sim \nu$.
We follow the usual convention and write $(\Omega, \mathcal{F},\mathbb{P}_\nu)$ and $\E_\nu$ for the underlying probability space and expectation, respectively, cf. \cite{douc2018markov} for details.

Given a probability measure $\mu$ on $(\cX, \mathcal{B}(\cX))$, we define $\mu P$ as 
\[
\mu P (\cdot ) = \int_\cX P(x, \cdot) \mu(\d x),
\] 
which is again a probability measure.  Recall that if $\mu P = \mu$, then $\mu$ is called \textit{invariant distribution}, which will be typically denoted by $\pi$.

For a sufficiently regular function $f\colon \cX \to \R$ we define 
\[
Pf (x) = \int_\cX f(x') P(x , \d x').
\]
For $p\in [1, \infty)$, we denote by $L^p(\mu)$ the set of functions $f \colon \cX \to \R$ for which 
\[
\norm{f}_{L^p(\mu)}^p = \int_\cX \vert f(x)\vert^p \mu(\d x) <\infty.
\]
Furthermore, define the subspaces $L_0^p(\mu) = \{ f \in L^2 \colon \mu(f)=0\}$ of centred $L^p$ functions. Any Markov kernel $P$ with invariant distribution $\pi$ defines a linear operator on $L^p(\pi)$ and
\[
\norm{P}_{L^p(\pi) \to L^p(\pi)} := \sup_{\norm{f}_{L^p(\pi) \leq 1}} \norm{P f}_{L^p(\pi)} = 1.
\]
We refer e.g. to \cite{rudolf2012explicit} for a proof.

\subsection{Wasserstein contractions}

Let $\mu_1, \mu_2 \in \mathcal{P(\cX)}$ and denote by $\mathcal{C}(\mu_1, \mu_2)$ the set of couplings between $\mu_1$ and $\mu_2$. The 1-Wasserstein distance, sometimes also called Kantorovich-Rubinstein metric, is defined as 
\[
\mathcal{W}(\mu_1, \mu_2) = \inf_{\eta \in \mathcal{C}(\mu_1, \mu_2)} \int_{\cX \times \cX} d(x,y) \eta(\d x , \d y).
\]
If $\mathcal{W}$ is finite, then it satisfies all axioms of a metric. It is possible to generalise the above definition to so-called $p$-Wasserstein distances, with $p \geq 1$, we refer to \cite{villani2009optimal}.
However, we only consider the 1-Wasserstein distance in this article, hence we shall simply call it \textit{Wasserstein distance}.

Define the \textit{Kantorovich norm} of a Markov kernel $P \colon \cX \times \mathcal{B}(\cX) \to [0,1]$ as 
\[
\tau(P) = \sup_{x \neq y} \frac{\mathcal{W}(P(x, \cdot ), P(y,\cdot ))}{d(x,y)}.
\]
This quantity was used e.g.~by Dobrushin \cite{dobrushin1956central}, see Equation 1.5 there, who later called it Kantorovich norm in \cite{dobrushin96}.
In \cite{douc2018markov} the terminology \textit{Dobrushin coefficient} is used and in \cite{rudolf2018pertubation} it was called \textit{generalised ergodicity coefficient}.
From \cite{ollivier2009ricci}, see also \cite{Joulin2010Curvature,paulin2016mixing}, it is known that $\tau(P)$ is connected to the Ricci curvature of the corresponding Markov chain.

\begin{defi}\label{def:contractive_chain}
A Markov chain $(X_n)_{n \in \N}$ with transition kernel $P$ is called \textit{Wasserstein contractive}, if $\tau(P)<\infty$ and $\tau(P^m)<1$ for some $m \in \N$. 
\end{defi}

\begin{rem}
From this point onwards, we always assume that $m \in \N$ in Definition \ref{def:contractive_chain} is the minimal $m \in \N$ such that $\tau(P^m)<1$.
\end{rem}

\begin{rem}
The assumption $\tau(P^m)<1$ is closely related to the so-called multistep coarse curvature of \cite{paulin2016mixing}.
\end{rem}
Let us briefly explain the terminology. Propositions 14.3 and 14.4 of \cite{dobrushin96}, see also \cite[Proposition 2.1]{rudolf2018pertubation}; \cite[Proposition 20.3.3]{douc2018markov}, yield 
\begin{itemize}
	\item[(i)] $\tau(P P') \leq \tau (P) \tau (P')$, and 
	\item[(ii)]$\mathcal{W}(\mu_1 P, \mu_2 P) \leq \tau(P) \mathcal{W}(\mu_1, \mu_2)$,
\end{itemize}
where $P, P'$ are Markov kernels and $\mu_1, \mu_2 \in \mathcal{P(\cX)}$.
By (ii), if $\tau(P^m)<1$, then $\mu \mapsto \mu P^m$ (considered as mapping from $\mathcal{P(\cX)}$ to $\mathcal{P(\cX)}$) is a contraction w.r.t.~the Wasserstein distance $\mathcal{W}$. 
Therefore, one might expect that there is a fixed point (aka invariant distribution), say $\pi$. 
Yet, $\mathcal{W}$ might be infinite in general. 
However, if $\int_\cX d(x_0, x) P(x_0, \d x) < \infty$ for some $x_0 \in \cX$, then we can prove the existence of $\pi$, c.f. \cite[Theorem 20.3.4]{douc2018markov}, see also \cite{paulin2016mixing}.

\begin{rem}\label{rem:def_Lambda}
From property (i) above it follows that for a contractive chain
\[
\sum_{n=0}^\infty \tau(P^n) = \sum_{\ell=0}^\infty\sum_{j=0}^{m-1}\tau(P^{m\ell +j}) \leq \frac{m \tau(P)^m}{1-\tau(P^m)} < \infty.
\]
Hence, for any contractive chain we may define $\Lambda := \sum_{n=0}^\infty \tau(P^n) \in [0, \infty)$.
\end{rem}

\subsection{Examples}
In this section we present examples of Wasserstein contractive Markov chains with emphasis on connections to MCMC algorithms.

\begin{example}[{Heat bath for the Ising model -- cf. \cite[Example 17]{ollivier2009ricci} and \cite[Section 2.2]{Joulin2010Curvature}}]
Here we set $\cX = \{-1,1\}^G$, for some finite graph $G$. We define $\pi$ to be the probability measure which has density 
\[
p(x) \propto \exp \left( - \beta\sum_{g_1 \sim g_2} x(g_1)x(g_2) - h \sum_{g \in G} x(g) \right) = \exp (-U(x)),
\]
where $\beta\geq 0$, $h \in \R$ and $g_1 \sim g_2$ means that $g_1$ and $g_2$ are connected. 
The normalisation constant is given by $Z = \sum_{x \in \cX} \exp(-U(x))$. 

Set $d(x_1,x_2) = \frac{1}{2}\sum_{g \in G} \vert x_1(g)- x_2(g) \vert$. Then, $d(x_1, x_2)$ measures to the number of $g \in G$ for which $x_1(g)$ and $x_2(g)$ are different.  
Moreover, for $x \in \cX$ and $g \in G$ we define $y_-\in \cX$ and $y_+ \in \cX$ by setting $y_-(g)=-1$ and $y_+(g)=1$, and $y_-(g') = y_+(g') = x(g')$ for any $g' \neq g$.

The heat bath, or Glauber dynamics, is a Markov chain $(X_n)_{n \in \N}$ where at each step a random element $g \in G$ is chosen. Then the value of $x(g)$ is updated to $-1$ or $1$ with probabilities (proportional to) $e^{-U(y_-)}$ and $e^{-U(y_+)}$, respectively. 
As shown in \cite{ollivier2009ricci, Joulin2010Curvature} for $\beta$ small enough (see \cite[Example 17]{ollivier2009ricci} for details) the chain is contractive with $\pi$ being the invariant distribution.
\end{example}

\begin{example}[Simple Slice Sampling]\label{ex:simple_slice_sampling}
Let $\cX = \R^s$ and let $\rho = d $ be the Euclidean metric.
Assume that the target distribution $\pi$ has a density which satisfies the conditions of \cite[Theorem 2.1]{natarovskii2021quantitative}. 
Roughly speaking, this means that (with additional minor technical requirements) the density of $\pi$ is rotationally invariant. 
We consider a Markov chain $(X_n)_{n \in \N}$ based on (simple) slice sampling, see \cite[Algorithm 1.1]{natarovskii2021quantitative}.
The main idea of the transition mechanism is as follows. 
Given $X_n = x$, one first samples $t$ uniformly from the set $[0, p(x)]$, where $p$ is the (non-normalised) density of $\pi$. Then, $X_{n+1}$ is sampled uniformly from the level set $\{y \in \cX \colon p(y) \geq t\}$.

Indeed, this algorithm targets $\pi$ as invariant distribution and by \cite[Theorem 2.1]{natarovskii2021quantitative} we have that 
\[
\frac{\mathcal{W}(P(x, \cdot ) , P(y, \cdot )) }{d(x,y)} \leq \left( 1- \frac{1}{s+1}\right),
\]
for any $x \neq y \in \R^s$,
where $P$ is the slice sampling kernel and $\mathcal{W}$ the Wasserstein distance corresponding to $d$. 
Hence, $\tau(P) \leq  1- (s+1)^{-1} <1$.
\end{example}

A class of Wasserstein contractive Markov chains are so-called \textit{uniformly ergodic} ones. This is summarised in the upcoming lemma, for a proof see e.g. \cite[Remark 2.1]{rudolf2018pertubation}.
\begin{lem}\label{lem:uniformly_ergodic_implies_contraction}
Let $(X_n)_{n \in \N}$ be a uniformly ergodic Markov chain, that is 
\[
\sup_{x \in \cX}\norm{P^k(x,\cdot) -\pi}_{tv} \leq C \kappa^k,
\]
for some $C\in (0, \infty)$ and $\kappa<1$. Then, $(X_n)_{n \in \N}$ is Wasserstein contractive for $\mathcal{W}$ based on the trivial metric $d (x,y) = 2\cdot\mathds{1}_{\{x \neq y\}}$.
\end{lem}
Uniformly ergodic Markov chains are of particular interest in the Markov chain Monte Carlo literature, as the upcoming three examples demonstrate. 

\begin{example}[Independent Metropolis-Hastings]\label{example:independent_Metropolis_Hastings}
Consider $(X_n)_{n \in \N}$ based on an {independent Metropolis Hastings} algorithm, which is a Markov chain Monte Carlo method to approximate target distributions $\pi$. 
Given $X_n$ a candidate move for $X_{n+1}$ is proposed (independently of $X_n$), and then accepted with a certain probability, see \cite[Algorithm 1]{wang2022exact} for details. 
If the ratio between the density of $\pi$ and the proposal density is uniformly bounded, then the chain is uniformly ergodic \cite{wang2022exact}, whence contractive by Lemma \ref{lem:uniformly_ergodic_implies_contraction}.
\end{example}

\begin{example}[Geodesic Slice Sampling]\label{ex:geodesic_slice_sampling}
We consider a sufficiently regular manifold $\mathscr{M}$ (see \cite{hasenpflug2025uniform} Assumption A and Theorem 1 for details). Geodesic slice sampling, see \cite{durmus2025geodesicslicesamplingriemannian}, is a  Markov chain Monte Carlo method to approximately sample from a distribution $\pi$ on $\mathscr{M}$. 
The main idea of the algorithm is similar as in Example \ref{ex:simple_slice_sampling}, however, certain steps are adapted to fit better into the manifold context, we refer to \cite{durmus2025geodesicslicesamplingriemannian}. 
In the setting of \cite[Theorem 1]{hasenpflug2025uniform} the corresponding Markov chain is uniformly ergodic, and therefore contractive by Lemma \ref{lem:uniformly_ergodic_implies_contraction}.
\end{example}

\begin{example}[Stereographic MCMC]\label{ex:stereographic_MCMC}
We consider stereographic Markov chain Monte Carlo algorithms, as recently introduced in \cite{yang2024stereographic}.
Let $\cX = \R^s$, and $p \colon \R^s \to \R^+$ be a probability density (w.r.t.~the Lebesgue measure) corresponding to $\pi$.
The main idea of stereographic MCMC is to project $p$ (and hence $\pi$) to the sphere.
In this way it is possible to profit from the geometry of the sphere and obtain chains that mix quickly. 
Since the stereographic projection is bijective, these properties are inherited when projecting back to $\R^s$. 
Indeed, if $p$ is continuous and its tails are not too heavy, then \cite[Theorem 2.1]{yang2024stereographic} shows the uniform ergodicity of their \textit{Stereographic Projection Sampler}. 
By Lemma \ref{lem:uniformly_ergodic_implies_contraction} the corresponding chain is also contractive. 
\end{example}

Another class of Markov chains for which we have a contraction are \textit{$V$-uniformly ergodic} chains. 
These fit well into a `drift and minorisation' framework, see Example \ref{ex:Lyapunov} below, and are also of interest in the MCMC literature. 
For examples of $V$-uniformly ergodic chains, including versions of the Metropolis-Hastings algorithms, see e.g.~\cite[Section 3.4]{roberts2004general}.

The upcoming result, which follows by combining {\cite[Lemma 2.1]{hairer2011yet} and the Kantorovich-Rubinstein duality formula, see e.g. \cite{villani2009optimal}, shows that $V$-uniformly ergodic chains are contractive. 
\begin{lem}\label{lem:V_uniformly_ergodic_implies_contraction}
Let $V \colon \cX \to [1, \infty)$ be $\pi$-integrable, lower semi-continuous and $(X_n)_{n \in \N}$ a $V$-uniformly ergodic Markov chain, that is, for any $k \in \N$,  
\[
\sup_{\vert f \vert \leq V} \left\vert P^kf(x) - \pi(f) \right\vert \leq V(x)C \kappa^k,
\]
for some $C\in (0, \infty)$ and $\kappa<1$ and any $x \in \cX$. Then,  $(X_n)_{n \in \N}$ is Wasserstein contractive for $\mathcal{W}$ based on the metric $d (x,y) = \mathds{1}_{\{x \neq y\}}(V(x) + V(y))$.
\end{lem}

\begin{example}\label{ex:Lyapunov}
	We consider an aperiodic and irreducible Markov chain $(X_n)_{n \in\N}$ with kernel $P$ and invariant distribution $\pi$. Assume there exist a lower semi-continuous and measurable function $V \colon \cX \to [1, \infty)$, a set $S \in \mathcal{F}_\cX$, a probability measure on $\mu$ on $S$, as well as numbers $\lambda, \delta \in (0,1)$ and $b \in (0, \infty)$ such that the following two conditions are satisfied:
	\begin{itemize}
		\item[1)] for any $x \in \cX $ we have $PV(x) \leq \lambda V(x) + b \mathds{1}_S(x)$ and $\sup_{x \in S} V(x) < \infty$;
		\item[2)] for any $x \in S$ we have $P(x, A) \geq\delta \mu(A \cap S)$. 
	\end{itemize}
	A function $V$ satisfying 1) is called \textit{Lyapunov function} in the Markov chain literature (see e.g. \cite{meyn2009markov,douc2018markov}). 
	It is known (see e.g.~the appendix of \cite{hofstadler2024convergenceRates}) that if $V$ is a Lyapunov function, then so is $V^q$, with $q \in (0,1]$, where in 1) $\lambda $ and $b$ get replaced by  $\lambda^q$ and $b^q$, respectively.  
	
	Let $q < (0,1]$. 
	Define the (lower semi-continuous) metric $d_q(x,y) = \mathds{1}_{\{x \neq y\}} (V^q(x) + V^q(y))$, and let $\mathcal{W}_q$ be the associated Wasserstein distance. 
	Since $V^q$ is also a Lyapunov function and additionally 2) is satisfied, the chain $(X_n)_{n \in \N}$ is $V^q$-uniformly ergodic, c.f.~\cite{Baxendale2005renewal}, and by Lemma \ref{lem:V_uniformly_ergodic_implies_contraction} contractive w.r.t.~$\mathcal{W}_q$.
\end{example}

\begin{example}[MALA]\label{ex:MALA}
Let $\cX = \R^s$, $\rho$ be the Euclidean metric. 
Let $U \colon \cX \to [0, \infty)$ and assume the probability measure $\pi$ has density 
\[
p(x) \propto \exp(-U(x)).
\]
The idea of the Metropolis adjusted Langevin algorithm (MALA) is to make use of the stochastic differential equation (SDE), 
\[
\d Y_t = - \nabla U(Y_t)\d t + \sqrt{2}\d W_t, 
\]
where $W_t$ is a Brownian motion. 
Roughly speaking, one uses an Euler-Maruyama discretisation of this SDE as proposal within a Metropolis-Hastings algorithm, see e.g. \cite{durmus2023MALA} for further details and historical remarks. 

If $U$ satisfies Assumption 1 and 3 of \cite{durmus2023MALA} and the discretisation step, say $\gamma>0$, is chosen sufficiently small, then \cite[Theorem 1]{durmus2023MALA} implies that the MALA kernel $P_\gamma$ is $V$-uniformly ergodic for 
\[
V(x) = \exp (\eta \norm{x}^2),
\]
with some $\eta >0$, where $\norm{\cdot}$ denotes the Euclidean norm. Hence, the corresponding chain is contractive by Lemma \ref{lem:V_uniformly_ergodic_implies_contraction}.
\end{example}

\begin{example}[No-U-Turn-Sampler]\label{ex:NUTS}
Let $\cX = \R^s$, $\rho$ be the Euclidean metric. 
Let $U \colon \cX \to [0, \infty)$ and assume the probability measure $\pi$ has density 
\[
p(x) \propto \exp(-U(x)).
\]
The No-U-Turn-Sampler (NUTS), see \cite{hoffman2014no}, is a variant of Hamiltonian Monte Carlo, and is currently one of the most popular MCMC algorithms. 
The main idea is to extend $(\cX, \pi)$ 
by adding so-called momentum variables, cf. \cite{hoffman2014no}. 
These  have a physical interpretation and can then be used to efficiently target an extended distribution $\widetilde{\pi}$ having $\pi$ as marginal.  

Recently, it was shown that under suitable assumptions (a variant of) NUTS is $V$-geometrically ergodic for $V(x) = \exp(\eta \norm{x})$, where $\norm{\cdot} $ is the Euclidean norm and $\eta>0$, cf. \cite[Theorem 17]{durmus2023convergence}.
Thus, the Markov chain corresponding to NUTS is contractive as a consequence of Lemma \ref{lem:V_uniformly_ergodic_implies_contraction}. 
\end{example}

\begin{rem}
In \cite{eberle2019quantitative} contractions for $\mathcal{W}$ are obtained by transforming the metric $\rho$ into a suitable distance $d$.
This allows to `choose' $d$ based on the Markov chain at hand, cf. Sections 2.1 and 2.2 there, and also Sections 2.4 -- 2.5 for applications to chains induced by Euler schemes, or the Metropolis adjusted Langevin algorithm, respectively. 
\end{rem}

\section{Poisson's equation}\label{sec:Solutions}
To employ the martingale decomposition technique \eqref{equ:martingale_decomposition}, we require solutions of Poisson's equation in the context of Wasserstein contractive chains. Emphasis is also on regularity of solutions which allows for more tractable constants in our main results later.

Throughout, let $(X_n)_{n \in \N}$ be a Markov chain evolving in $\cX$ with transition kernel $P\colon \cX \times \mathcal{B}(\cX) \to [0,1]$ and unique invariant distribution $\pi$.

\subsection{Solutions for Lipschitz functions} 
Recall that $f$ is Lipschitz w.r.t.~$d$ if there exists a constant $L\geq0$ such that $\vert f(x) - f(y) \vert \leq L \cdot  d(x,y)$ for any $x,y \in \cX$. We always refer to the metric $d$ when speaking of Lipschitz functions and denote the space of all measurable Lipschitz functions on $\cX$ by $\mathscr{L}$.
The mapping $\norm{\cdot}_d \colon \mathscr{L} \to [0, \infty)$, given by 
\[
\norm{f}_d = \sup_{x \neq y} \frac{\vert f(x) - f(y)\vert}{d(x,y)},
\] 
is a semi-norm on $\mathscr{L}$, cf. \cite{cobzacs2019lipschitz}. 
Note that $\norm{f}_d=0$ if $f$ is constant. 
Hence, $(\mathscr{L}, \norm{\cdot}_d)$ is not a normed space in general. 

The Kantorovich-Rubinstein duality formula, see \cite[Chapter 5]{villani2009optimal}, connects $(\mathscr{L}, \norm{\cdot}_d)$ and $\mathcal{W}$ via 
\[
\mathcal{W}(\mu_1, \mu_2) = \sup_{f \in B_1(\mathscr{L})} \left\vert \int_\cX f(x) \mu_1(\d x) - \int_\cX f(x) \mu_2(\d x) \right\vert,
\]
with $B_1(\mathscr{L}) = \{f \in \mathscr{L} \colon \norm{f}_d \leq 1\}$ being the unit ball in $\mathscr{L}$.

To study Poisson's equation in the context of Lipschitz functions it is vital that $\pi(f)$ is well-defined for any $f \in \mathscr{L}$. Following \cite{Joulin2010Curvature} we define the $p$-eccentricity $E_p \colon \cX \to [0,\infty]$, with $p\in [1, \infty)$, by
\[
E_p(x) = \int_\cX d(x,y)^p \pi(\d y).
\]
If $E(x_0)< \infty$ for some $x_0 \in \cX$, then $E(x)<\infty$ for any $x \in \cX$ by triangle inequality; compare to \cite{Joulin2010Curvature}.
\begin{lem}\label{lem:integrability_Lipschitz}
Assume that $E_p(x_0)< \infty$ for some $x_0 \in \cX$ and $p \in [1, \infty)$. Then, for any $f \in \mathscr{L}$ we have $\pi(\vert f \vert^p ) < \infty$, and consequently also $\pi(f)$ is well-defined and finite.
\end{lem}

\begin{proof}
For any $f \in \mathscr{L}$ we note that $\vert f \vert^p $ is non-negative and measurable. Observe that $\vert f(x) \vert^p \leq (\vert f(x)- f(x_0) \vert + \vert f(x_0) \vert)^p \leq (d(x,x_0) + \vert f(x_0)\vert)^p$.  
It follows that 
\[
\int_\cX \vert f(x) \vert^p \pi(\d x) \leq  \int_\cX (d(x, x_0) + \vert f(x_0) \vert)^p \pi(\d x)
\leq 2^p\left(  E_p(x_0) + \vert f(x_0) \vert^p\right),
\]
where in the last step we used the inequality $(\alpha+\beta)^p \leq 2^p (\alpha^p+\beta^p)$, which is valid for any $\alpha, \beta \geq 0$.
The statement about $\pi(f)$ is now a simple consequence of Hoelder's inequality (consider $\pi(f \cdot \mathds{1})$).
\end{proof}

Under the conditions of Lemma \ref{lem:integrability_Lipschitz} the right hand side of Poisson's equation \eqref{equ:Poisson_equation}, i.e.~$f(x)-\pi(f)$, is well-defined for any $f \in \mathscr{L}$. 
In this case it is no loss of generality to assume $ f \in \mathscr{L}_0 = \{f \in \mathscr{L} \colon \pi(f)=0\}$ which allows to work in a Banach space, compare \cite{cobzacs2019lipschitz}. We include a proof for the convenience of the reader. 

\begin{lem}\label{lem:L_0_Banachspace}
Assume that $E_p(x_0) < \infty$ for some $x_0 \in \cX$ and $p \in [1,\infty)$. Then, $(\mathscr{L}_0, \norm{\cdot}_d)$ is a Banach space. 
\end{lem}

\begin{proof}
Let $\xi \in \cX$ be fixed and set $\mathscr{L}_\xi = \{ f \in \mathscr{L} \colon f (\xi)=0\}$.
First we show that $(\mathscr{L}_\xi, \norm{\cdot}_d)$ is a Banach space. 
Namely, let $(f_n)_{n \in \N}$ be a Cauchy sequence in $(\mathscr{L}_\xi, \norm{\cdot}_d)$. Then, there exists a Lipschitz function $f$ with $f(\xi)=0$ such that $\norm{f_n -f}_d \to 0$, see \cite[Theorem 8.1.3]{cobzacs2019lipschitz}.
Additionally, 
\[
\vert f_n(x) - f(x) \vert = \vert f_n(x) - f(x) - (f_n(\xi) -f(\xi)) \vert \leq \norm{f_n - f}_d d(x, \xi),
\]
for any $x \in \cX$. 
Hence, $\lim_{n \to \infty}f_n (x) = f(x)$ for any $x \in \cX$. Thus $f$ is measurable and therefore $f \in \mathscr{L}_\xi$. 

Next we show that $(\mathscr{L}_\xi, \norm{\cdot}_d )$ and $(\mathscr{L}_0, \norm{\cdot}_d)$ are isometrically isomorph. 
To this end, observe that by Lemma~\ref{lem:integrability_Lipschitz} any $f \in \mathscr{L}_\xi$ is $\pi$-integrable.  
We define the mappings $T \colon \mathscr{L}_0 \to \mathscr{L}_\xi$ and $R \colon \mathscr{L}_\xi \to \mathscr{L}_0$ via 
\[
T(f) = f(\cdot) - f(\xi)  \qquad\text{and} \qquad R(h) = h(\cdot) - \pi(h). 
\] 
Note that both, $T$ and $R$, are linear. Furthermore,  
\[
\norm{Tf }_d = \norm{f}_d \qquad\text{and}\qquad \norm{Rh}_d= \norm{h}_d,
\]
and $T (R(h)) = h$ as well as $R(T(f)) = f$ as required.
\end{proof}

In what follows we always equip $\mathscr{L}_0$ with the norm $\norm{\cdot}_d$. For the sake of brevity, we simply write $\mathscr{L}_0$ for $(\mathscr{L}_0, \norm{\cdot}_d)$.
The upcoming lemma relates $\tau(K)$ to the operator norm of (the Markov operator) $K$, when acting on $\mathscr{L}_0$. A similar statement as in Lemma \ref{lem:operator_norm_Lipschitz} is also in an early version of \cite{rudolf2018pertubation}. The final (published) version, however, does not use that result. 

\begin{lem}\label{lem:operator_norm_Lipschitz}
Assume that $E_p(x_0) < \infty$ for some $x_0 \in \cX$; $p \in [1,\infty)$, and that the Markov kernel $K \colon \cX \times\mathcal{B}(\cX) \to [0,1]$ satisfies $\tau(K)< \infty$.
Then, the mapping $K \colon \mathscr{L}_0 \to \mathscr{L}_0$ given by $f \mapsto Kf$ defines a linear operator and we have 
\[
\norm{K}_{\mathscr{L}_0 \to \mathscr{L}_0} = \tau(K). 
\]  
\end{lem}

\begin{proof}
First, we note that if $\tau(K)<  \infty$, then $K \colon \mathscr{L} \to \mathscr{L}$ defines a linear (and bounded) operator.
By invariance, the same is true for $K \colon \mathscr{L}_0 \to \mathscr{L}_0$.

Observe that for $f \in \mathscr{L}$ and $c \in \R$ we have $\norm{f}_d = \norm{f +c}_d$.
Particularly, setting $c= \pi(f)$ yields 
\[
\sup_{f \in B_1(\mathscr{L}_0)}\norm{K f}_d = \sup_{f \in B_1(\mathscr{L})}\norm{K f}_d,
\]
where $B_1(\mathscr{L})$ and $B_1(\mathscr{L}_0)$ are the unit balls in $\mathscr{L}$ and $\mathscr{L}_0$, respectively. 
Finally, by using the Kantorovich-Rubinstein duality formula, it is possible to show $ \sup_{f \in B_1(\mathscr{L})}\norm{K f}_d = \tau(K)$, compare to \cite[Chapter 20.3]{douc2018markov}.
\end{proof}

Next we compute solutions of Poisson's equation for on $\mathscr{L}_0$.

\begin{thm}\label{thm:Solution_Poisson_Lipschitz_general}
Assume that $E_p(x_0) < \infty$ for some $x_0 \in \cX$ and $p \in [1,\infty)$ and let $(X_n)_{n \in \N}$ be a Wasserstein contractive Markov chain. 
Then, for any $f \in \mathscr{L}_0$ the space $\mathscr{L}_0$ contains exactly one solution $u = u_f$ of Poisson's equation \eqref{equ:Poisson_equation} with forcing function $f$ and 
\begin{itemize}
	\item[i)] $\norm{u}_d \leq \Lambda \norm{f}_d$, as well as
	\item[ii)] $\int_\cX \vert u(x) \vert^p \pi(\d x) \leq 2^p \inf_{y \in \cX}(\vert u(y)\vert^p + \norm{u}_d^p E_p(y)) < \infty$.
\end{itemize}
\end{thm}   

\begin{proof}
Note that all assumptions to apply Lemmas \ref{lem:integrability_Lipschitz}, \ref{lem:L_0_Banachspace} and \ref{lem:operator_norm_Lipschitz} are satisfied, such that 
\[
\norm{P^m}_{\mathscr{L}_0 \to \mathscr{L}_0} = \tau(P^m) <1.
\]
Since $\mathscr{L}_0$ is a Banach space, the Neumann series corresponding to $P$, i.e. $\sum_{k=0}^\infty P^k$, converges. Thus $(\Id-P) \colon \mathscr{L}_0 \to \mathscr{L}_0$ is invertible and  
\[
(\Id - P)^{-1} = \sum_{k=0}^\infty P^k. 
\]
Additionally, $\norm{(\Id - P)^{-1}}_{\mathscr{L}_0 \to \mathscr{L}_0} \leq \Lambda$, with $\Lambda < \infty$ as in Remark \ref{rem:def_Lambda}. 

Given $f \in \mathscr{L}_0$ we observe that for any $u \in \mathscr{L}_0$ we have
\[
(\Id - P)u = f\quad \Longleftrightarrow\quad u = (\Id-P)^{-1}f.
\]
Whence, in the space $\mathscr{L}_0$ the unique solution to Poisson's equation (with forcing function $f$) is given by $u = (\Id-P)^{-1}f$. 

For the rest of the proof fix $f\in \mathscr{L}_0$ and $u= (\Id-P)^{-1}f$.
To see the bound on $\norm{u}_d$ claimed in i), we note that 
\[
\norm{u}_d \leq \norm{(\Id - P)^{-1}}_{\mathscr{L}_0 \to \mathscr{L}_0} \norm{f}_d \leq \Lambda \norm{f}_d.
\]
Observing $E_p(x_0)< \infty$, triangle inequality yields that $E_p(x)< \infty$ for any $x \in \cX$. By the same steps as in the proof of Lemma \ref{lem:integrability_Lipschitz}, we deduce that 
\[
\int_\cX \vert u(x) \vert^p \pi(\d x) \leq 2^p (\vert u(y)\vert^p + \norm{u}_d^p E_p(y))
\]
for arbitrary $y \in \cX$, which implies ii). 
\end{proof}

Theorem \ref{thm:Solution_Poisson_Lipschitz_general} establishes the existence of solutions to Poisson's equation. The integrability property ii), however, only provides an upper bound involving function evaluations of $u$.
Under slightly stronger assumptions on the eccentricity it is possible to obtain a refined bound. 
\begin{thm}\label{thm:Solution_Poisson_with_L_p_bound}
Let the setting be as in Theorem \ref{thm:Solution_Poisson_Lipschitz_general}, and additionally assume that $\int_\cX \vert E_1(x)\vert^{p_0} \pi(\d x) < \infty$ for some $p_0\in [1, \infty)$. 
Then, for any $f \in \mathscr{L}_0$ the space $\mathscr{L}_0$ contains exactly one solution $u = u_f$ of Poisson's equation with forcing function $f$ which, additionally to the conclusions of Theorem \ref{thm:Solution_Poisson_Lipschitz_general}, satisfies
\[
\int_\cX \vert u(x) \vert^{p_0} \pi(\d x) \leq \left( \Lambda \cdot \norm{f}_d \right)^{p_0} \int_\cX \vert E_1(x)\vert^{p_0} \pi(\d x). 
\]
\end{thm}

\begin{proof}
From the proof of Theorem \ref{thm:Solution_Poisson_Lipschitz_general} we know that $u $ is given by the Neumann series
\[
u(x) = (\Id-P)^{-1} [f]  (x) = \sum_{k=0}^\infty P^k f(x).
\]
Note that $\vert P^k f(x)\vert \leq \norm{f}_d \mathcal{W}(P^k(x, \cdot), \pi) \leq \norm{f}_d E_1(x) \tau(P^k)$, as a consequence of the Kantorovich-Rubinstein duality formula. Hence 
\begin{equation}\label{equ:pointwise_bound_u}
\vert u(x) \vert^{p_0} \leq \left(\Lambda \cdot \norm{f}_d \right)^{p_0} E_1(x)^{p_0}. 
\end{equation}
Combining \eqref{equ:pointwise_bound_u} with the assumptions on $E_1$ implies the desired result. 
\end{proof}

Using \eqref{equ:pointwise_bound_u} we also obtain the following `$L^\infty$-version' of Theorem \ref{thm:Solution_Poisson_with_L_p_bound}.

\begin{cor}
Let the setting be as in Theorem \ref{thm:Solution_Poisson_Lipschitz_general}, and additionally assume that $E_1(x) \leq c < \infty$ for ($\pi$-almost) all $x \in \cX$, with some $c \in (0, \infty)$. Then, for any $f \in \mathscr{L}_0$ the space $\mathscr{L}_0$ contains exactly one solution $u = u_f$ of Poisson's equation with forcing function $f$ which, additionally to the conclusions of Theorem \ref{thm:Solution_Poisson_Lipschitz_general}, for ($\pi$-almost) all $x \in \cX$ satisfies
\[
\vert u(x) \vert \leq c \Lambda \norm{f}_d.
\]
\end{cor}

\begin{rem}
	Solvability of Poisson's equation for uniformly, or $V$-uniformly ergodic Markov chains were computed in \cite{GLYNN2002hoeffding} and \cite{glynn1996liapounov}, respectively. 
	Lemma~\ref{lem:uniformly_ergodic_implies_contraction} and Lemma~\ref{lem:V_uniformly_ergodic_implies_contraction} show that it is possible to recover these settings as special cases of Wasserstein contractive chains. More recently, the authors of \cite{glynn2025computableboundssolutionpoissons} studied Poisson's equation and provided computable bounds for $u$ under such conditions.
	In particular, \cite[Equation (2.10) and Theorem 2.1]{glynn2025computableboundssolutionpoissons} establish a probabilistic representation of $u$.
	Similar bounds and connections to perturbation theory of Markov chains/kernels were also shown in \cite{herve2025computable}, where a different proof technique was used.
\end{rem}

\subsection{Solutions for $L^p$ functions}

We conclude this section with Poisson's equation in the context of $L^2$-spectral gaps. We show existence of $u$ and provide a bound for the $L^p(\pi)$ norm in terms of an $L^2$-spectral gap, see Theorem~\ref{thm:Solution_Poisson_with_L_p_bound}. These bounds may be useful for extending our main results to the $L^2$ case. However, a thorough analysis is beyond the scope of the present paper and is left for future research. 

Following \cite{rudolf2012explicit} we say that a Markov kernel $P$ with invariant distribution $\pi$ admits an $L^2$-\textit{spectral gap} if $\norm{P}_{L^2_0(\pi) \to L^2_0(\pi)}<1$. Wasserstein contractions and $L^2$-spectral gaps are connected: Assuming reversibility, a Wasserstein contraction implies an $L^2$-spectral gap if, bounded Lipschitz functions are dense in $L^2(\pi)$, see \cite[Proposition 2.8]{hairer2014spectral}; or we have $d=\rho$ and finite \textit{coarse diffusion} $\sigma$ cf. \cite[Definition 18]{ollivier2009ricci} and \cite[Proposition 30]{ollivier2009ricci}.

We begin with the following auxiliary result, which is a simple combination of \cite[Lemma 3.16]{rudolf2012explicit} and \cite[Proposition 3.17]{rudolf2012explicit}. 

\begin{lem}\label{lem:spectral_gap_implies_L_p_bound}
For the Markov kernel $P$ we have the following bounds
\[
\norm{P^\ell}_{L^p_0(\pi) \to L^p_0(\pi)} \leq 
\begin{cases}
	2^{2/p} \norm{P}_{L^2_0(\pi) \to L^2_0(\pi)}^{2\ell \frac{p-1}{p}} \qquad &\text{ if } p\in (1,2),\\[1.2ex]
	2^{2\frac{p-1}{p}} \norm{P}_{L^2_0(\pi) \to L^2_0(\pi)}^{2\ell/p} &\text{ if } p \in (2,\infty),
\end{cases}
\]
for any  $\ell \in \N$. 
\end{lem} 

\begin{proof}
	For any $p \in (1, \infty)$ define $\Pi \colon L^p(\pi) \to L^p(\pi)$ via $f \mapsto \pi(f)$. 
	Let $\ell \in \N$. 
	Combining \cite[Lemma 3.16]{rudolf2012explicit} and \cite[Proposition 3.17]{rudolf2012explicit} we obtain $	\norm{P^\ell}_{L^p_0(\pi) \to L^p_0(\pi)} \leq \norm{P^\ell-\Pi}_{L^p(\pi) \to L^p(\pi)}$ as well as
	\[
 \norm{P^\ell-\Pi}_{L^p(\pi) \to L^p(\pi)}
	\leq \begin{cases}
		2^{2/p} \norm{P}_{L^2_0(\pi) \to L^2_0(\pi)}^{2\ell \frac{p-1}{p}} \qquad &\text{ if } p\in (1,2),\\
		2^{2\frac{p-1}{p}} \norm{P}_{L^2_0(\pi) \to L^2_0(\pi)}^{2\ell/p} &\text{ if } p \in (2,\infty),
	\end{cases}
	\]
	which yields the desired result. 
\end{proof}

The following theorem shows that of Poisson's equation is solveable in the $L^p_0$ if we have an $L^2$-spectral gap. In \cite{douc2018markov} the existence of $u$ is shown assuming that $\sum_{\ell=0}^\infty \norm{P^\ell f}_{L^p(\pi)}$ converges. We show that an $L^2$ spectral gap is a sufficient condition for any $p \in (1,\infty)$, and provide $L^p(\pi)$ bounds for $u$ with constants only depending on the spectral gap and $p$. 

\begin{thm}\label{thm:poisson_equation_spectral_gap_general}
Let $p \in (1,\infty)$ and $f \in L^p_0(\pi)$. 
If the Markov kernel $P$ admits a spectral gap, then, the space $L^p_0(\pi)$ contains exactly one solution of Poisson's equation $u \in L^p_0(\pi)$ (with forcing function $f$) and we have 
\[
\norm{u}_{L^p(\pi)}
\leq 
\begin{cases}
	\frac{2^{2/p}}{1- \kappa^{2\frac{p-1}{p}}} \norm{f}_{L^p(\pi)} \qquad&\text{ if } p \in (1,2), \\[1.5ex]
	\frac{1}{1-\kappa} \norm{f}_{L^2(\pi)} &\text{ if } p = 2, \\[1.2ex]
	\frac{2^{2\frac{p-1}{p}}}{1-\kappa^{2/p}}\norm{f}_{L^p(\pi)} &\text{ if } p\in (2, \infty),
\end{cases}
\] 
where $\kappa = \norm{P}_{L^2_0(\pi) \to L^2_0(\pi)}< 1$. 
\end{thm}

\begin{proof}
We split the proof into three cases, dealing with different $p \in(1, \infty)$.

\textbf{Case 1:} Let $p \in (1,2)$. Since $\kappa = \norm{P}_{L^2_0(\pi) \to L^2_0(\pi)} <1$ we have $\kappa^{2\frac{p-1}{p}} <1$. 
By Lemma \ref{lem:spectral_gap_implies_L_p_bound},  
\[
\norm{P^\ell}_{L^p_0(\pi) \to L^p_0(\pi)} \leq 2^{2/p} \left(\kappa^{2\frac{p-1}{p}} \right)^\ell.
\]
Hence, on $L^p_0(\pi)$, we have $(\Id - P)^{-1} = \sum_{\ell=0}^\infty P^\ell$. Additionally we have the following bound for the operator norm, 
\[
\norm{(\Id - P)^{-1}}_{L^p_0(\pi) \to L^p_0(\pi)} \leq \frac{2^{2/p}}{1- \kappa^{2\frac{p-1}{p}}}.
\]
Now, $ u = (\Id - P)^{-1} f \Leftrightarrow u - Pu =f$, such that Poisson's equation (with forcing function $f$), admits the unique solution $u$ in $L^p_0(\pi)$, and 
\[
\norm{u}_{L^p(\pi)} \leq \norm{(\Id - P)^{-1}}_{L^p_0(\pi) \to L^p_0(\pi)} \norm{f}_{L^p(\pi)} \leq \frac{2^{2/p}}{1- \kappa^{2\frac{p-1}{p}}} \norm{f}_{L^p(\pi)}.
\] 

\textbf{Case 2:} Let $p=2$. 
Setting $u = (\Id - P)^{-1}f$ we obtain the desired result by convergence of the Neumann series.

\textbf{Case 3:} Let $p >2$. Here we can use the same arguments as in case 1, but with the corresponding bounds for $p >2$ from Lemma \ref{lem:spectral_gap_implies_L_p_bound}. 
\end{proof}

\section{Maximal inequalities}\label{sec:max_inequalities}
In this section we apply our results to derive maximal inequalities for Wasserstein contractive Markov chains.
These serve as main tool to derive the pathwise stability results later. 

Throughout this section $(X_n)_{n \in \N}$ is a Wasserstein contractive chain with transition kernel by $P$. The invariant and initial distribution are $\pi$ and $\nu$, respectively. For integrable $f \colon \cX \to \R$ with $\pi(f)=0$ and $n \in \N$ we define 
\[
S_n^* f := \max_{1 \leq k \leq n} \vert S_k f \vert = \max_{1 \leq k \leq n} \left\vert \sum_{j=1}^k f(X_j) \right\vert.
\]

If $u $ is a solution of Poisson's equation with forcing function $f$, then by using \eqref{equ:martingale_decomposition}, for any $1 \leq k \leq n$, we obtain 
\begin{equation}\label{equ:maximal_decomposition}
S_n^*f \leq M_n^* + R_n^*,
\end{equation}
with $M_n^* = \max_{1 \leq k \leq n} \vert M_k\vert$ and $R_n^* = \max_{1 \leq k \leq n} \vert u(X_1) - u(X_{k+1}) \vert $.

\subsection{Regularity of the martingale $M_n$}
Here we investigate the integrability conditions of $(M_n)_{n \in \N}$. Recall that the underlying probability space is $(\Omega, \mathcal{F}, \mathbb{P}_\nu)$ and that
the filtration $(\mathcal{F}_n)_{n \in \N}$ is defined as $\mathcal{F}_n = \sigma(X_k \colon 1\leq k \leq n+1)$.

Let $\nu \ll \pi$ with Radon-Nikod\'ym derivative $\frac{\d \nu}{\d \pi} \in L^q(\pi)$, for some $q \in (1,\infty]$.
For some generic integrable $f \colon \cX \to \R$ assume $u$ is the solution of Poisson's equation with forcing function $f$. 
Throughout this section $M_n, R_n$ (see \eqref{equ:martingale_decomposition}) correspond to this $f$ and $u$. 

\begin{lem}\label{lem:martingale_bound_1}
Let $s$ be the Hoelder conjugate of $q \in (1, \infty)$.
Let $p\geq 1$ and let $r \geq sp$. If $u \in L^{r}(\pi)$, then $(M_n)_{n \in \N}$ is a martingale w.r.t.~$(\mathcal{F}_n)_{n \in \N}$ and for its differences we have
\[
\E_\nu[ \vert M_{n+1} - M_n \vert^p ]^{1/p} \leq 2\norm{\frac{\d \nu}{\d \pi}}_{L^q(\pi)}^{1/p} \norm{u}_{L^r(\pi)}
\]
for any $n \in \N$.
In particular, this implies that $(M_n)_{n \in \N}$ is a $L^p$ martingale. 
\end{lem}

\begin{proof}
We begin by showing the claimed bound on the differences. 
For the sake of brevity write $\Delta_j = u(X_{j+1}) -Pu(X_j) $ such that $M_n = \sum_{j=1}^n \Delta_j$. 

Note that $\mathbb{P}_\nu$ and $\mathbb{P}_\pi$ only differ by the choice of the initial distribution. 
By Hoelder's inequality
\begin{align*}
\E_\nu\left[\vert \Delta_j\vert^p\right]^{1/p} = \E_\pi\left[ \frac{\d \nu}{\d \pi}(X_1) \vert \Delta_j\vert^p\right]^{1/p} 
\leq\norm{\frac{\d \nu}{\d \pi}}_{L^q(\pi)}^{1/p} \E_\pi\left[  \vert \Delta_j\vert^{ps}\right]^{1/(sp)}.
\end{align*}
By assumption $r\geq sp$, such that by definition of $\Delta_j$ we obtain
\begin{align*}
\E_\pi\left[  \vert \Delta_j\vert^{ps}\right]^{1/(sp)}&\leq 
\E_\pi\left[  \vert \Delta_j\vert^{r}\right]^{1/r}\\ 
&\leq 
\E_\pi\left[\vert u(X_{j+1}) \vert^r \right]^{1/r} + \E_\pi\left[\vert Pu(X_{j}) \vert^r \right]^{1/r} 
\leq 2 \norm{u}_{L^r(\pi)},
\end{align*}
where in the last step we used that $\norm{P}_{L^r(\pi) \to L^r(\pi)} =1$. 

It follows that $\E_\nu[\vert M_n \vert] < \infty$ for any $n\in \N$. Moreover, by definition of $M_n$, we also have 
\[
\E [\Delta_n \vert \mathcal{F}_{n-1}] = 0
\]
showing that $(M_n)_{n \in \N}$ is a martingale as claimed. 
\end{proof}

The proof of the previous lemma also transfers to the case $q= \infty$. 
We summarise this in the following lemma, whose proof is left to the reader.
\begin{lem}\label{lem:martingale_bound_2}
Assume $q=\infty$ and let $p \geq 1$. 
If $u \in L^{p}(\pi)$, then $(M_n)_{n \in \N}$ is a martingale w.r.t.~$(\mathcal{F}_n)_{n \in \N}$ and for its differences we have
\[
\E_\nu[ \vert M_{n+1} - M_n \vert^p ]^{1/p} \leq 2\norm{\frac{\d \nu}{\d \pi}}_{L^\infty(\pi)}^{1/p} \norm{u}_{L^p(\pi)}
\]
for any $n \in \N$.
In particular, this implies that $(M_n)_{n \in \N}$ is a $L^p$ martingale. 
\end{lem}

\subsection{Maximal inequalities for Lipschitz functions}
In this section we show bounds for $S_n^* f$, with $f \in \mathscr{L}_0$. For $p \in (1, \infty)$, define 
\[
\eps_p  = \norm{E_1(\cdot )}_{L^p(\pi)} = \left( \int_\cX E_1(x)^p \pi(\d x) \right)^{1/p} \in [0, \infty].  
\]
Note that if $\eps_p < \infty $, then also $E_1(x_0) < \infty$ for some $x_0 \in \cX$.

In order prove our main result we make use of the following condition. 
\begin{ass}\label{ass:bounded_R_n_star}
	
	There exists some $\widehat{C} >0$ such that for any forcing function $f \in \mathscr{L}_0$ we have
	\[
	\E_\nu \left[ \left\vert R_n^*  \right\vert^2  \right] \leq \widehat{C} n \norm{f}_d^2
	\] 
	for any $n \in \N$, with $R_n^*$ as in \eqref{equ:maximal_decomposition}.  
\end{ass}

\begin{thm}\label{thm:Doob_maximal_inequality_finite_moment_assumption}
	Let $f \in \mathscr{L}_0$ and $\frac{\d \nu}{\d \pi} \in L^q(\pi)$ for some $q \in (1, \infty)$, with Hoelder conjugate $s \in (1, \infty)$. 
	Assume that {$\eps_r<\infty$ for $r=2s$}.
	If Assumption~\ref{ass:bounded_R_n_star} is satisfied, then, for any $n \in \N$  we have 
	\[
	\sup_{f \in \mathscr{L}_0; \norm{f}_d \leq 1}
	\mathbb{P}_\nu\left[  S_n^*f > t \right] \leq \frac{4n}{t^2}\left(4 \norm{\frac{\d \nu}{\d \pi}}_{L^q(\pi)} \left(\Lambda \eps_r\right)^2 + \widehat{C} \right),
	\]
	for any $t>0$.
\end{thm}

\begin{proof}
	Let $f \in \mathscr{L}_0$ with $\norm{f}_d\leq 1$ be fixed. 
	Due to Theorem \ref{thm:Solution_Poisson_with_L_p_bound} the corresponding solution of Poisson's equation $u$ is in $L^r_0(\pi)$.   
	Lemma \ref{lem:martingale_bound_1} implies that the martingale $(M_n)_{n \in \N}$, as specified in \eqref{equ:martingale_decomposition}, has increments, say $(\Delta_j)_{j \in \N}$, with $\Delta_j = u(X_{j+1}) -Pu(X_j) $, which are uniformly bounded in $L^2$, that is 
	\[
	\E_\nu\left[\Delta_j^2\right] \leq  4\norm{\frac{\d \nu}{\d \pi}}_{L^q(\pi)} \norm{u}_{L^{r}(\pi)}^2.
	\]
	Doob's maximal inequality (see Corollary 1.6 of Chapter II in \cite{revuz1999continuous}) yields 
	\[
	\mathbb{P}_\nu\left[ M_n^* > \delta \right] \leq \frac{1}{\delta^2} \E_\nu\left[ M_n^2 \right] = \frac{1}{\delta^2}\E_\nu\left[ \sum_{j=1}^n \Delta_j^2 \right] \leq \frac{4n\norm{\frac{\d \nu}{\d \pi}}_{L^q(\pi)} \norm{u}_{L^{r}(\pi)}^2}{\delta^2},
	\]
	for any $\delta>0$.	Hence, by combining Assumption \ref{ass:bounded_R_n_star} and $\norm{f}_d\leq 1$ we get 
	\[
	\mathbb{P}_\nu\left[ R_n^* > \delta \right] \leq \frac{1}{\delta^2} \E_\nu\left[ (R_n^*)^2 \right] \leq \frac{\widehat{C} \cdot n}{\delta^2}. 
	\]
	The statement follows by observing that 
	\[
	\mathbb{P}_\nu\left[ S_n^*f > t \right] \leq \mathbb{P}_\nu\left[ M_n^* > t/2 \right] + \mathbb{P}_\nu\left[ R_n^* > t/2 \right],
	\]
	and that $\norm{u}_{L^r(\pi)} \leq \Lambda \eps_r $ (see Theorem \ref{thm:Solution_Poisson_with_L_p_bound}). 
\end{proof}

Combining the bounds for $\E_\nu\left[(M_n)^2\right]$ and $\E_\nu\left[ (R_n^*)^2 \right]$ from the proof of Theorem \ref{thm:Doob_maximal_inequality_finite_moment_assumption} with Doob's $L^2$-inequality, cf. \cite{revuz1999continuous}, we obtain the following result. 
\begin{cor}\label{cor:L_2_maximal_inequality_finite_moment_assumption}
	Under the same assumptions as in Theorem \ref{thm:Doob_maximal_inequality_finite_moment_assumption}, and with $q,r, \widehat{C}$ as specified there, for any $n \in \N$ we have 
	\[
	\sup_{f \in \mathscr{L}_0; \norm{f}_d \leq 1}
	\E_\nu[(S_n^* f)^2] \leq 
	{4n}\left(8 \norm{\frac{\d \nu}{\d \pi}}_{L^q(\pi)} \left(\Lambda\eps_r\right)^2 + \widehat{C} \right).
	\]
	
\end{cor}

We conclude this section by showing that Assumption \ref{ass:bounded_R_n_star} is satisfied for the setting considered in Example \ref{ex:Lyapunov}.

\begin{prop}\label{prop:remainder_assumption_suff_condition}
Let $(X_n)_{n \in \N}$ be a Wasserstein contractive Markov chain. Assume $\sup_{j \in \N}\E_\pi[d(X_1, X_{j+1})^{2\beta}]< \infty$ and $\frac{\d \nu}{\d \pi} \in L^{\beta'}(\pi)$, where $\beta, \beta' \in (1,\infty)$ are Hoelder conjugate. Then, Assumption \ref{ass:bounded_R_n_star} is satisfied. 
\end{prop}

\begin{proof}
		Let $n \in \N$ be fixed and for $1 \leq j \leq n$ choose pairwise disjoint $A_j^{(n)} \in \mathcal{F}$ such that on each $A_j^{(n)}$ we have $R_n^* = \vert u(X_1) - u(X_{j+1}) \vert$. 
	
	It follows that 
	\begin{align*}
		\E_\nu\left[(R_n^*)^2\right] &= \sum_{j=1}^n \E_\nu\left[ \mathds{1}_{A_j^{(n)}} \vert u(X_1) - u(X_{j+1})\vert^2  \right] \\
		&\leq 
		\sum_{j=1}^n \E_\nu\left[ \vert u(X_1) - u(X_{j+1})\vert^{2}  \right].
	\end{align*}
	By Theorem \ref{thm:Solution_Poisson_Lipschitz_general} the function $u$ is Lipschitz with constant $C(\tau) \norm{f}_d = {C'}$, where $C(\tau)$ does not depend on $f \in \mathscr{L}_0$. Hence, by Hoelder's inequality
	\begin{align*}
	\E_{\nu}[\vert u(X_1) - u(X_{j+1}) \vert^2] &\leq (C')^2 \E_{\nu} [d(X_1, X_{j+1})^2]  \\
	&\leq 
	(C')^2 \E_{\pi}[d(X_1, X_{j+1})^{2\beta}]^{\frac{1}{\beta}} \norm{\frac{\d \nu}{\d \pi}}_{L^{\beta'}(\pi)}.
	\end{align*}
This implies the claimed result. 
\end{proof}

\begin{rem}\label{rem:sufficient_conditions_remainder_assumption}
The condition $\sup_{j \in \N}\E_{\pi}[d(X_1, X_{j+1})^{2\beta}]<\infty$ in Proposition~\ref{prop:remainder_assumption_suff_condition} is satisfied for example if:
\begin{itemize}
	\item[1)] The diameter of $\cX$ is bounded. 
	\item[2)] We are in the `drift and minorisation' setting of Example \ref{ex:Lyapunov} and 
	\[
	d (x,y) = d_{\frac{1}{2\beta}}(x,y) = \mathds{1}_{\{x \neq y\}} \left( V(x)^{\frac{1}{2\beta}} + V(y)^{\frac{1}{2\beta}}\right). 
	\]
	\item[3)] The metric $d$ is the Euclidean metric, and $\E[\norm{X}^{2\beta}]< \infty$, where $X\sim \pi$ and $\norm{\cdot }$ is the Euclidean norm. 
\end{itemize}
Hence, in those settings Assumption \ref{ass:bounded_R_n_star} is true if additionally $\frac{\d \nu}{\d \pi} \in L^{\beta'}(\pi)$. 
\end{rem}

\section{Pathwise stability of contractive chains}\label{sec:path_wise}
This section contains our main results about the pathwise stability of contractive Markov chains as well as applications to concrete MCMC algorithms. 

\subsection{Convergence results}\label{subsec:path_wise_convergence_theory}

We begin with the following result, which establishes a rate of convergence for single realisations of a contractive chain $(X_n)_{n \in \N}$. The assumptions are the same as in Theorem \ref{thm:Doob_maximal_inequality_finite_moment_assumption}. 

\begin{thm}\label{thm:path_wise_convergence_general}
	Let $f \in \mathscr{L}_0$ and $\frac{\d \nu}{\d \pi} \in L^q(\pi)$ for some $q \in (1, \infty)$, with Hoelder conjugate $s \in (1, \infty)$. 
Assume that {$\eps_r<\infty$ for $r=2s$}.
If Assumption~\ref{ass:bounded_R_n_star} is satisfied, then, there exists $C \colon \Omega \to (0, \infty)$ such that for $\mathbb{P}_\nu$-almost all $\omega \in \Omega$ and any $n \in \N$ we have 
\[
\left\vert \frac{1}{n}\sum_{j=1}^n f(X_j(\omega)) - \pi(f) \right\vert \leq C(\omega) \frac{\log n}{\sqrt{n}} .
\] 
\end{thm}

\begin{proof}
	Set $n_k = 2^k$ and $r_k = \sqrt{k}\log (k)$ for $k \in \N$. Theorem \ref{thm:Doob_maximal_inequality_finite_moment_assumption} implies that 
	\[
	\mathbb{P}_\nu\left[ S_{n_k}^*f > r_{n_k} t \right] \leq \frac{c}{(tk)^2},
	\]
	for any $t >0$, with some absolute constant $c>0$.
	Employing the lemma of Borel-Cantelli we deduce that $\mathbb{P}_\nu$-almost surely
	\[
	\lim_{k \to \infty}
	\frac{1}{r_{n_k}}S_{n_k}^* f  =0.
	\] 
	Now, given $n \in\N$, chose $k$ such that $n_{k-1} < n \leq n_k$. 
	We observe 
	\[
	0 \leq \left\vert \frac{1}{r_n} S_n f \right\vert \leq \frac{1}{r_n} S_{n_k}^* f \leq \frac{4}{r_{n_k}} S_{n_k}^*f.
	\]
	Hence
	\[
	\lim_{n \to \infty} \left\vert \frac{1}{r_n} S_n f \right\vert =0,
	\]
	$\mathbb{P}_\nu$-almost surely implying the desired result. 
\end{proof}

The following results allows to control the random constant on an event of `high probability'.  

\begin{thm}\label{thm:almost_unif_conve_good_rate}
Let $f \in \mathscr{L}_0$ and $\frac{\d \nu}{\d \pi} \in L^q(\pi)$ for some $q \in (1, \infty)$, with Hoelder conjugate $s \in (1, \infty)$. Assume that {$\eps_r<\infty$ for $r=2s$}. If Assumption~\ref{ass:bounded_R_n_star} is true, then, there exists $c \in (0, \infty)$ such that for any $N \in \N\setminus\{1\}$, we have
\[
\mathbb{P}_{\nu} \left[ \forall n \geq N \colon \left\vert \frac{S_nf}{n} \right\vert \leq \frac{\log(n)}{\sqrt{n}}  \right] \geq 1 - \frac{c}{\log(N)}.
\]
\end{thm} 

Before we turn to the proof, let us some add the following consequence. 
\begin{cor}\label{cor:almost_unif_conv_good_rate}
Under the assumptions of Theorem \ref{thm:almost_unif_conve_good_rate}, for any $f \in \mathscr{L}$ and $N \in \N\setminus\{1\}$, there exists a measurable set $A\subseteq \Omega$ such that $\mathbb{P}[A] \geq 1-c/\log(N)$ we have   
\[
\left\vert \frac{1}{n}\sum_{j=1}^{n} f(X_j(\omega)) - \pi(f) \right\vert \le \frac{\log(n)}{\sqrt{n}} 
\]
for any $\omega \in A$ and $n \geq N$. 
\end{cor}

\begin{proof}[Proof of Theorem \ref{thm:almost_unif_conve_good_rate}]
Note that all requirements to apply Theorem \ref{thm:Doob_maximal_inequality_finite_moment_assumption} are satisfied. We define the increasing mapping $\Phi (n) = \sqrt{n} \log(n)$ Then, a union bound and Theorem \ref{thm:Doob_maximal_inequality_finite_moment_assumption} show
\begin{align}\label{equ:almost_uniform_convergence_inequality}
&\mathbb{P}_{\nu} \left[ \exists n \geq N \colon \left\vert S_nf \right\vert > \Phi(n) \right]	\\
&\qquad\leq 
\sum_{j=0}^\infty \mathbb{P}_{\nu} \left[  \exists n \in \{2^jN, \dots, 2^{j+1}N-1\} \colon \vert S_n f\vert > \Phi(n) \right] \nonumber \\
&\qquad\leq 
\sum_{j=0}^\infty \mathbb{P}_{\nu} \left[  \exists n \in \{2^jN, \dots, 2^{j+1}N-1\} \colon \vert S_n f\vert > \Phi(2^jN) \right] \nonumber \\
&\qquad\leq 
\sum_{j=0}^\infty \mathbb{P}_{\nu} \left[ \max_{2^jN \leq n < N2^{j+1}} \colon \left\vert {S_nf} \right\vert > \Phi(2^jN) \right] 
\leq 
\sum_{j=0}^\infty \frac{\norm{f}^2_d\widetilde{C}2^{j+1}N}{\Phi(2^jN)^2} \nonumber,
\end{align}
where $\widetilde{C}$ is the constant appearing in Theorem~\ref{thm:Doob_maximal_inequality_finite_moment_assumption}.
By definition of $\Phi(n)$ and Lemma \ref{lem:estimate_series}, which we postpone to the end of this section, we have 
\[
\sum_{j=0}^\infty \frac{2^{j+1}N}{\Phi(2^jN)^2} \leq \sum_{j=0}^\infty \frac{2}{\log(2^j N)^2} \leq \frac{\pi}{\log(2)\log(N)}.
\]
Thus, by setting $c = \frac{\pi \widetilde{C}}{\log(2)}$, it follows that 
\[
\mathbb{P}_{\nu} \left[ \forall n \geq N \colon \left\vert S_nf \right\vert \leq \Phi(n) \right]	= 1 - \mathbb{P}_{\nu} \left[ \exists n \geq N \colon \left\vert S_nf \right\vert > \Phi(n) \right]	\geq 1 - \frac{c}{\log(N)}.
\]
Observing that $\left\vert S_nf \right\vert \leq \Phi(n)$ if and only if $\vert n^{-1}S_n f \vert \leq \frac{\log(n)}{\sqrt{n}}$ finishes the proof. 
\end{proof}

\begin{rem}\label{rem:computation_of_constant}
	The constant $c$ in Theorem \ref{thm:almost_unif_conve_good_rate} can, in principle, be computed. Namely, the estimate \eqref{equ:almost_uniform_convergence_inequality} and Theorem \ref{thm:Doob_maximal_inequality_finite_moment_assumption} reveal
	\[
	c = \frac{\pi \widetilde{C} \norm{f}_d^2}{\log(2)} = \frac{4\pi\norm{f}_d^2}{\log(2)} \cdot \left(4 \norm{\frac{\d \nu}{\d \pi}}_{L^q(\pi)} \left(\Lambda \eps_r\right)^2 + \widehat{C} \right),
	\]
	with $\widehat{C}$ as in Assumption~\ref{ass:bounded_R_n_star}. Taking into account Remark \ref{rem:sufficient_conditions_remainder_assumption}, bounds for $\frac{\d \nu}{\d \pi}$, $\eps_r$, $\Lambda$ and $\E_{\pi}[d(X_1, X_{j+1})^{2\beta}]$ would then allow to also control $c$. 
	
	Furthermore, we note that $\frac{\d \nu}{\d \pi}$ appears in $\widetilde{C}$ and the bound for $\widehat{C}$ given by Proposition \ref{prop:remainder_assumption_suff_condition}. Thus, a $M$-warm start, in which case $\frac{\d \nu}{\d \pi}$ is bounded by $M>0$, may help to reduce the constant. 
\end{rem}

\begin{rem}\label{rem:tractable_bound_for_probability}
Using $\Phi(n) = \sqrt{c} \log(n)\sqrt{n}$ in the proof of Theorem \ref{thm:almost_unif_conve_good_rate}, where $c$ is as in Remark \ref{rem:computation_of_constant}, it is also possible to show 
\[
\mathbb{P}_{\nu} \left[ \forall n \geq N \colon \left\vert \frac{S_nf}{n} \right\vert \leq \frac{\sqrt{c}\log(n)}{\sqrt{n}}  \right] \geq 1 - \frac{1}{\log(N)}.
\]
In this way the probability in Theorem \ref{thm:almost_unif_conve_good_rate} and Corollary \ref{cor:almost_unif_conv_good_rate} is bounded in a somewhat more tractable way. 
\end{rem}

Although $1- c/{\log(N)} \to 1$ as $N\to \infty$, this convergence is rather slow. The following result addresses this issue by allowing for slower convergence rates of $n^{-1}S_n f$. 

\begin{thm}\label{thm:almost_unif_conve_slow_rate}
Let $f \in \mathscr{L}_0$ and $\frac{\d \nu}{\d \pi} \in L^q(\pi)$ for some $q \in (1, \infty)$, with Hoelder conjugate $s \in (1, \infty)$. Assume that {$\eps_r<\infty$ for $r=2s$}. If Assumption~\ref{ass:bounded_R_n_star} is true, then, for $N \in \N$ and $\alpha \in (0, 1/2)$ there exists a constant $C=C_\alpha$ such that we have
\[
\mathbb{P}_{\nu} \left[ \forall n \geq N \colon \left\vert \frac{S_nf}{n} \right\vert \leq \frac{1}{n^{1/2-\alpha}}  \right] \geq 1 - \frac{C}{N^{2\alpha}}.
\]
\end{thm}  

\begin{proof}
Set $\Phi(n) = n^{1/2+\alpha}$. The same arguments that were used to derive \eqref{equ:almost_uniform_convergence_inequality} in the proof of Theorem \ref{thm:almost_unif_conve_good_rate} show
\[
\mathbb{P}_{\nu} \left[ \exists n \geq N \colon \left\vert S_nf \right\vert > \Phi(n) \right]	
\leq 
\sum_{j=0}^\infty \frac{\widetilde{C}\norm{f}_d^22^{j+1}N}{\Phi(2^jN)^2}
= \frac{C}{N^{2\alpha}}.
\]
The proof is finished with the same steps as in the proof of Theorem \ref{thm:almost_unif_conve_good_rate}.
\end{proof}

\begin{cor}\label{cor:almost_unif_conv_slow_rate}
	Under the assumptions of Theorem \ref{thm:almost_unif_conve_good_rate}, for any $f \in \mathscr{L}$, $N \in \N$ and $\alpha \in (0,1/2)$, there exists a measurable set $A\subseteq \Omega$ such that $\mathbb{P}[A] \geq 1-C/N^{2\alpha}$ we have   
	\[
	\left\vert \frac{1}{n}\sum_{j=1}^{n} f(X_j(\omega)) - \pi(f) \right\vert \le \frac{1}{{n^{1/2-\alpha}}} 
	\]
	for any $\omega \in A$ and $n \geq N$. 
\end{cor}

Also in the setting of Theorem \ref{thm:almost_unif_conve_slow_rate} and Corollary \ref{cor:almost_unif_conv_slow_rate} versions of Remark \ref{rem:computation_of_constant} and \ref{rem:tractable_bound_for_probability} can be deduced.

\begin{lem}\label{lem:estimate_series}
For any $N \in \N\setminus\{1\}$ we have 
\[
\sum_{j=0}^\infty \frac{1}{\log(N2^j)^2} \leq \frac{\pi}{ \log(2) \log(N)}. 
\] 
\end{lem}
\begin{proof}
Note that 
\[
\frac{1}{\log(N2^j)^2} = \frac{1}{(j\log(2) + \log(N))^2} \leq \frac{1}{j^2\log(2)^2 + \log(N)^2}.
\]
Now, for $T \in \N$, using the substitution $u = \frac{\log(2) x}{\log(N)}$, we estimate
\begin{align*}
\sum_{j=1}^T\frac{1}{j^2\log(2)^2 + \log(N)^2} &\leq \int_0^{T+1} \frac{1}{\log(2)^2 x^2 + \log(N)^2} \; \d x \\ 
&= \frac{\arctan\left( \frac{\log(2)(T+1)}{\log(N)} \right) }{\log(2)\log(N)} \leq \frac{\pi}{2 \log(2)\log(N)}.
\end{align*}
Taking $T\to \infty$ it follows 
\[
\sum_{j=0}^\infty \frac{1}{\log(N2^j)^2} \leq \frac{1}{\log(N)^2} + \sum_{j=1}^\infty \frac{1}{j^2\log(2)^2 + \log(N)^2}  \leq  \frac{\pi}{ \log(2) \log(N)}. 
\]
\end{proof}

\subsection{The median trick}

Next we consider the so-called \textit{median trick}, compare to \cite{niemiro2009fixed, latuszynski2011rigorous}, which is a powerful technique to amplify properties which hold with large probability. Let $\ell \in \N$ be odd, and $(X_n^1)_{n \in \N}, \dots, (X_n^\ell)_{n \in \N}$ be independent copies of the Wasserstein contractive chain $(X_n)_{n \in \N}$. Given some measurable Lipschitz function $f$, we define  
\[
M_{n,\ell} f  = \text{median}\left( \frac{1}{n}\sum_{j=1}^{n}f(X_j^1) , \dots, \frac{1}{n}\sum_{j=1}^{n}f(X_j^\ell) \right).
\]

\begin{thm}\label{thm:probability_bound_median}
Let $f \in \mathscr{L}_0$ and $\frac{\d \nu}{\d \pi} \in L^q(\pi)$ for some $q \in (1, \infty)$, with Hoelder conjugate $s \in (1, \infty)$. Assume that {$\eps_r<\infty$ for $r=2s$}. If Assumption~\ref{ass:bounded_R_n_star} is true, then, for any $N \in \N$ with $N\geq 8$, we have
\[
\mathbb{P}_{\nu} \left[ \forall n \geq N \colon \left\vert M_{n,\ell} f \right\vert \leq \frac{\sqrt{c}\log(n)}{\sqrt{n}}  \right] \geq 1 - \exp\left( -2\ell \left(\frac{1}{2} - \frac{1}{\log(N)}\right)^2 \right),
\]
where $c >0$ can be chosen as in Remark \ref{rem:computation_of_constant}.
\end{thm}

Before we turn to the proof let us add and discuss the following consequence.
\begin{cor}\label{cor:median_as_convergence}
	Under the assumptions of Theorem \ref{thm:probability_bound_median}, for any $f \in \mathscr{L}$, $N \in \N$ with $N \geq 8$, there exists a measurable set $A\subseteq \Omega$ such that $\mathbb{P}[A] \geq 1- \exp\left( -2\ell \left(\frac{1}{2} - \frac{1}{\log(N)}\right)^2 \right)$ and we have   
\[
\left\vert M_{n,\ell}(f)(\omega) - \pi(f) \right\vert \le \frac{\sqrt{c}\log(n)}{\sqrt{n}} 
\]
for any $\omega \in A$ and $n \geq N$, where $c>0 $ can be chosen as in Remark \ref{rem:computation_of_constant}. 
\end{cor} 

Note that already for small values of $\ell $ and $N$ the lower bound for $\mathbb{P}_\nu[A]$ in Corollary~\ref{cor:median_as_convergence} is close to $1$. This renders the above result useful in applications where it is important to achieve a given accuracy. For example, choosing $N=1000$ and $\ell=15$ and setting $a=1/2 - 1/\log(1000) $ yields 
\[
\mathbb{P}_\nu[A] \geq 1 - e^{-30a^2} \approx 0.977.
\]
Thus, with probability higher than $0.97$, the proxy $M_{n,\ell}f$ achieves the deterministic rate of convergence $\frac{\sqrt{c}\log(n)}{\sqrt{n}}$ if at least $N \cdot \ell = 15 000$ Markov steps are used. Additionally, if $\eps>0$ and $n_0 \geq N$ such that $\eps^{-2} \leq n_0/(c\log(n_0)^2)$, then, for $\omega \in A$, single runs of $M_{n,\ell}f$, satisfy
\[
\vert M_{n,\ell}f (\omega) - \pi(f) \vert \leq \eps,
\] 
for any $n \geq n_0$. Thus, for $n \geq \max\{n_0, N\}$, on the set $A$ single realisations of $M_{n,\ell}f$ have reached a given (deterministic) level of accuracy and, despite the stochasticity, also remain below the threshold $\eps$.  

\begin{proof}[Proof of Theorem \ref{thm:probability_bound_median}]
For the sake of brevity set $\frac{1}{n}\sum_{j=1}^{n}f(X_j^i) = \frac{1}{n}S_n^i f$ and $r_n = \frac{\sqrt{c} \log(n)}{\sqrt{n}}$. If 
\begin{equation}\label{equ:single_realisation_inequality}
\forall n \geq N \colon \left\vert\frac{1}{n}S_n^if \right\vert \leq r_n
\end{equation}
for at least $\frac{\ell+1}{2}$ indices $i \in \{1, \dots, \ell\}$, then by definition of the median we also have $\vert M_{n,\ell}f \vert \leq r_n$ for any $n\geq N$. Thus, 
\begin{align*}
&\mathbb{P}_\nu\left[ \forall n \geq N \colon \left\vert M_{n,m} f \right\vert \leq \frac{\sqrt{c}\log(n)}{\sqrt{n}}  \right] \\
&\geq 
\mathbb{P}_\nu \left[\eqref{equ:single_realisation_inequality} \text{ holds for at least } \frac{\ell+1}{2} \text{ indices } i \in \{1, \dots, m\}\right] := p_1.
\end{align*}
By independence and Remark \ref{rem:tractable_bound_for_probability}, $p_1$ can be estimated by considering $m$ independent Bernoulli trials with success probability $\delta = \delta_N = \frac{1}{\log(N)}$. That is, 
\[
p_1 \geq \sum_{j=\frac{\ell+1}{2}}^\ell {\ell\choose{j}} (1-\delta)^{j}\delta^{\ell-j} = 1- \sum_{j=\frac{\ell+1}{2}}^\ell {\ell\choose{j}} (1-\delta)^{\ell-j}\delta^{j}, 
\]
where the last inequality uses the fact that instead we can also consider the `fail probability' $\delta$ together with a symmetry argument. Let $(B_{i})_{i=1}^\ell$ be independent Bernoulli distributed with parameter $\delta$. Since $\log(N) > 2$, we get $\frac{1}{2} - \delta >0$ and  Hoeffding's inequality implies
\[
\sum_{j=\frac{\ell+1}{2}}^\ell {\ell\choose{j}} (1-\delta)^{\ell-j}\delta^{j}= \mathbb{P}\left[\sum_{i=1}^{\ell} B_i \ge \frac{\ell+1}{2}\right] \leq \exp\left( -2\ell\left(\frac{1}{2} - \delta \right)^2  \right).
\]
Hence 
\[
p_1 \geq 1- \exp\left( -2\ell\left(\frac{1}{2} - \delta \right)^2  \right),
\]
which shows the desired result. 
\end{proof}

\subsection{Applications}\label{subsec:applications}
Here demonstrate the applicability of our results to different MCMC algorithms. For the sake of exposition we assume that $\frac{\d \nu}{\d \pi} \in L^\infty(\pi)$. 

\subsubsection{Uniform ergodicity}
We consider uniformly ergodic MCMC methods, i.e. $(X_n)_{n \in \N}$ is a uniformly ergodic Markov chain. Examples that fit into this setting are, for instance, geodesic slice sampling in the setting of \cite{hasenpflug2025uniform} (see Example \ref{ex:geodesic_slice_sampling}), or the stereographic MCMC algorithms of \cite{yang2024stereographic} (see Example \ref{ex:stereographic_MCMC}), or the independent Metropolis algorithm (for suitably chosen proposal), see e.g.~\cite{wang2022exact} (see Example \ref{example:independent_Metropolis_Hastings}). 

Set $d(x,y)= 2 \cdot \mathds{1}_{\{x\neq y\}}$. Then, $\mathcal{W}$ coincides with the total variation distance, and $(X_n)_{n \in \N}$ is contractive (see \cite{rudolf2018pertubation} and Lemma~\ref{lem:uniformly_ergodic_implies_contraction}). By our choice of $d$ it follows $E_p(x) \leq 2^p$ and $\eps_p \leq 2$ for any $x \in \cX$ and $p \in [1, \infty)$. Furthermore, using Remark~\ref{rem:sufficient_conditions_remainder_assumption} we also have that Assumption~\ref{ass:bounded_R_n_star} is satisfied. Hence the results of the previous section are applicable. 

Thus, for any measurable and bounded $f \colon \cX \to \R$ and any $N \in \N$ for the convergence of $\frac{1}{n}\sum_{j=1}^n f(X_j(\omega))$ to $\pi(f)$ we have 
\begin{itemize}
\item[1)] the rate $\frac{C(\omega)\log(n)}{\sqrt{n}}$ for $\mathbb{P}_\nu$-almost all $\omega \in \Omega$, see Theorem \ref{thm:path_wise_convergence_general},
\item[2)] the rate $\frac{\log(n)}{\sqrt{n}}$ for $n \geq N$ and $\omega \in A$ with  $\mathbb{P}_\nu[A] \geq 1-\frac{c}{\log(N)}$, see Corollary~\ref{cor:almost_unif_conv_good_rate},
\item[3)] the rate $\frac{1}{n^{1/2-\alpha}}$ for $n \geq N$, $\alpha \in (0,1/2)$ and $\omega \in A$ with $\mathbb{P}_\nu[A] \geq 1- \frac{c}{N^{2\alpha}}$, see Corollary~\ref{cor:almost_unif_conv_slow_rate}.
\end{itemize}
Additionally, for the median estimator $M_{n,\ell}f$, using $m$ (odd) iid copies of $\frac{1}{n}S_n$, we obtain the rate $\frac{\log(n)}{\sqrt{n}}$ for any $n \geq N\geq 8$ and $\omega \in A$ with $\mathbb{P}_\nu[A] \geq1- \exp\left( -2\ell \left(\frac{1}{2} - \frac{1}{\log(N)}\right)^2 \right)$.

\subsubsection{Drift and minorisation conditions}
We consider chains $(X_n)_{n \in \N}$ for which we have drift and minorisation conditions as in Example~\ref{ex:Lyapunov}. MCMC methods that fit into this setting are MALA (compare to \cite[Theorem 1]{durmus2023MALA} see also Example~\ref{ex:MALA}), or NUTS (compare to \cite[Theorem 17]{durmus2023convergence} see also Example~\ref{ex:NUTS}). Reviewing the proofs in \cite{durmus2023MALA} and \cite{durmus2023convergence}, we see that the Lyapunov functions are given by $V(x) = \exp(\gamma \norm{x}^2)$ for MALA, and $V(x) = \exp(\gamma \norm{x})$ for NUTS, respectively, where $\gamma >0$.

Let $q \in (0,1/2)$ and $V$ a (measurable) Lyapunov function for $(X_n)_{n \in \N}$. Set 
\[
d(x,y) = d_q(x,y) = \mathds{1}_{\{x \neq y\}}(V^q(x) + V^q(y)).
\]
Recall that the corresponding Wasserstein distance is then closely related to the $V$-uniformly ergodic setting, compare to Lemma \ref{lem:V_uniformly_ergodic_implies_contraction}. In particular, the chain $(X_n)_{n \in \N}$ is contractive. Integrability of $V$ w.r.t.~$\pi$, compare e.g.~to \cite[Proposition 4.24]{hairer2006ergodic}, implies finiteness of $E_p(x)$ and $\eps_p$ for $p \in (2, 1/q]$. 

Similarly, by setting $p= 2\beta$ for suitable $\beta>1$, using invariance of $\pi$ the fact that $pq\leq 1$ and $V\geq 1$ we get 
\[
\E[d(X_1, X_{j})^{2\beta}] \leq c_p \left(\E_\pi[V^{pq}(X_1)] + \E_{\pi}[V^{pq}(X_j)]\right) \leq 2c_p \pi(V) < \infty,
\] 
uniformly in $j \in \N$. By Remark~\ref{rem:sufficient_conditions_remainder_assumption} it follows that Assumption \ref{ass:bounded_R_n_star} is true. 

Thus, the results of Section \ref{subsec:path_wise_convergence_theory} are applicable and for any measurable $f \colon \cX \to \R$ such that $\vert f \vert \leq V^q$ (implying that $f$ is Lipschitz w.r.t. $d_q$) we obtain the following statements regarding convergence of $\frac{1}{n}\sum_{j=1}^nf(X_j(\omega))$ to $\pi(f)$: For any $N \in \N$ we have 
\begin{itemize}
	\item[1)] the rate $\frac{C(\omega)\log(n)}{\sqrt{n}}$ for $\mathbb{P}_\nu$-almost all $\omega \in \Omega$, see Theorem \ref{thm:path_wise_convergence_general},
\item[2)] the rate $\frac{\log(n)}{\sqrt{n}}$ for $n \geq N$ and $\omega \in A$ with  $\mathbb{P}_\nu[A] \geq 1-\frac{c}{\log(N)}$, see Corollary~\ref{cor:almost_unif_conv_good_rate},
\item[3)] the rate $\frac{1}{n^{1/2-\alpha}}$ for $n \geq N$, $\alpha \in (0,1/2)$ and $\omega \in A$ with $\mathbb{P}_\nu[A] \geq 1- \frac{c}{N^{2\alpha}}$, see Corollary~\ref{cor:almost_unif_conv_slow_rate}.
\end{itemize}
Additionally, for the median estimator $M_{n,\ell}f$, using $\ell$ (odd) iid copies of $\frac{1}{n}S_n$, we obtain the rate $\frac{\log(n)}{\sqrt{n}}$ for any $n \geq N\geq 8$ and $\omega \in A$ with $\mathbb{P}_\nu[A] \geq1- \exp\left( -2\ell \left(\frac{1}{2} - \frac{1}{\log(N)}\right)^2 \right)$.

\subsubsection{Euclidean metric}
We consider a chain $(X_n)_{n \in \N}$ that is contractive for $\mathcal{W}$ based on the Euclidean metric. This is for instance the case for simple slice samplers in the setting of \cite[Theorem 2.1]{natarovskii2021quantitative} (see also Example \ref{ex:simple_slice_sampling}). Additionally, we assume that for $X \sim\pi$ we have $\E[\norm{X}^{2\beta}] < \infty$, for some $\beta>1$ with $\norm{\cdot}$ being the Euclidean norm. This implies (pointwise) finiteness of the eccentricity $E_{2\beta}$ as well as $\eps_{2\beta} < \infty$. Furthermore, by Remark \ref{rem:sufficient_conditions_remainder_assumption} Assumption \ref{ass:bounded_R_n_star} is satisfied.

It follows that all requirements to apply the results of Section \ref{subsec:path_wise_convergence_theory} are applicable. Hence, for any measurable Lipschitz function $f\colon \cX \to \R$ we obtain the following statements regarding convergence of $\frac{1}{n}\sum_{j=1}^nf(X_j(\omega))$ to $\pi(f)$: For any $N \in \N$ we have 
\begin{itemize}
	\item[1)] the rate $\frac{C(\omega)\log(n)}{\sqrt{n}}$ for $\mathbb{P}_\nu$-almost all $\omega \in \Omega$, see Theorem \ref{thm:path_wise_convergence_general},
\item[2)] the rate $\frac{\log(n)}{\sqrt{n}}$ for $n \geq N$ and $\omega \in A$ with  $\mathbb{P}_\nu[A] \geq 1-\frac{c}{\log(N)}$, see Corollary~\ref{cor:almost_unif_conv_good_rate},
\item[3)] the rate $\frac{1}{n^{1/2-\alpha}}$ for $n \geq N$, $\alpha \in (0,1/2)$ and $\omega \in A$ with $\mathbb{P}_\nu[A] \geq 1- \frac{c}{N^{2\alpha}}$, see Corollary~\ref{cor:almost_unif_conv_slow_rate}, and
\end{itemize}
Additionally, for the median estimator $M_{n,\ell}f$, using $\ell$ (odd) iid copies of $\frac{1}{n}S_n$, we obtain the rate $\frac{\log(n)}{\sqrt{n}}$ for any $n \geq N\geq 8$ and $\omega \in A$ with $\mathbb{P}_\nu[A] \geq1- \exp\left( -2\ell \left(\frac{1}{2} - \frac{1}{\log(N)}\right)^2 \right)$.

\section{Discussion}
We established convergence rates for realisations, i.e.~single outputs of MCMC algorithms. Under a Wasserstein contraction assumption we proved that by setting $r_n= \frac{\log(n)}{\sqrt{n}}$ or $r_n = n^{-1/2+\alpha}$, with $\alpha \in (0, 1/2)$, the estimator $\frac{1}{n}S_nf$ converges with rate $r_n$ on a set of `high probability' (depending on an initial stabilisation period). In particular, our results are applicable to several popular MCMC methods including NUTS, MALA and slice samplers. Furthermore, by employing the `median trick' it is possible to increase the probability on which our deterministic rates hold.

Some questions remain open: We showed in Section \ref{sec:Solutions} that Poisson's equation is also solveable under a $L^2$-spectral gap assumption. Hence it is natural to ask, whether our main results also hold in this setting. However, we leave this as potential future work. Another aspect, which seems more challenging, is the following: The lower bound for the probabilities in Theorems~\ref{thm:almost_unif_conve_good_rate} and \ref{thm:almost_unif_conve_slow_rate} depends on the choice of $r_n$. Our results indicate that there is a trade-off, that is, a slower rate $r_n$ allows for more control over $C(\omega)$. It would be interesting whether there are `optimal rates' (w.r.t.~a suitable criterion).

\section*{Acknowledgements}
JH is support by the EPSRC programme grant MaThRad EP/W026899/2. 

The author thanks Sam Power, Daniel Rudolf and Jere Koskela for pointing out useful links to literature, and helpful comments regarding the constants appearing in Section~\ref{sec:Wasserstein_contractive_chains} as well as the presentation of the paper.

During the preparation of this paper GenAI (ChatGPT) was used to search for relevant literature and to simplify and check some calculations. All ideas and writing are original.

\newcommand{\etalchar}[1]{$^{#1}$}

\end{document}